\documentclass{article}

\usepackage[english]{babel}

\usepackage[letterpaper,top=2cm,bottom=2cm,left=3cm,right=3cm,marginparwidth=1.75cm]{geometry}

\usepackage{amsmath}
\usepackage{stmaryrd}
\usepackage{graphicx}
\usepackage[colorlinks=true, allcolors=blue]{hyperref}
\usepackage{amsthm}
\usepackage{amssymb}
\usepackage{amsfonts}
\usepackage{tikz-cd}
\usepackage{enumitem}

\newtheorem{thm}{Theorem}[section]
\newtheorem{prop}[thm]{Proposition}
\newtheorem{cor}[thm]{Corollary}
\newtheorem{lemma}[thm]{Lemma}
\newtheorem{coj}[thm]{Conjecture}

\theoremstyle{definition}
\newtheorem{defn}[thm]{Definition} 
\newtheorem{ex}[thm]{Example} 

\theoremstyle{remark}
\newtheorem{rmk}[thm]{Remark}

\DeclareMathOperator{\spec}{Spec}
\DeclareMathOperator{\ext}{Ext}
\DeclareMathOperator{\trdeg}{tr.deg.}
\DeclareMathOperator{\can}{can}
\DeclareMathOperator{\Sp}{Sp}
\DeclareMathOperator{\ccd}{ccd}

\title{On the cohomological dimension of Siegel modular varieties and the modularity of formal Siegel modular forms}
\author{Haocheng Fan}
\date{}
\begin{document}
\maketitle

\begin{abstract}
We show that the coherent cohomological dimension of the Siegel modular variety $A_{g,\Gamma}$ (of genus $g\geq2$ and level $\Gamma$) is at most $\frac{1}{2}g(g+1)-2$. As a corollary, we show that the boundary of the compactified Siegel modular variety satisfies the Grothendieck-Lefschetz condition. This implies, in particular, that formal Siegel modular forms of genus $g\geq2$ are automatically classical Siegel modular forms. Our result generalizes the work of Bruinier and Raum on the modularity of formal Siegel modular forms in \cite{bruinier2015kudla} and \cite{bruinier2024formal}.
\end{abstract}

\section{Introduction}

In Kudla's famous paper \cite{kudla1997algebraic}, he conjectured the modularity of the generating series of special cycles. There are fruitful results in this direction, for example, assuming absolutely convergence of the generating series, Wei Zhang proved the modularity conjecture in the Siegel case in \cite{zhang2009modularity} and Yifeng Liu in the unitary case in \cite{liu2012arithmetic}. To remove the convergence condition, Bruinier and Raum proved the modularity of formal Fourier--Jacobi series in the Siegel case in \cite{bruinier2015kudla} and Jiacheng Xia in some unitary case in \cite{xia2022some} by using technical analytic methods. However, recently Bruinier and Raum presented a new approach in this automatic convergence problem in the Siegel case in \cite{bruinier2024formal}, which is closely related to our results on the coherent cohomological dimension of Shimura varieties. Actually, in an ongoing project joint with Wenxuan Qi, Linli Shi, Peihang Wu, Liang Xiao and Yichao Zhang, we have a new approach of the modularity conjecture, where the cohomological dimension appears naturally.

In this paper, we study the coherent cohomological dimension of Siegel modular varieties and its implications for the modularity of formal Siegel modular forms. We hope to extend these results to more general Shimura varieties in a future paper joint with Peihang Wu. Roughly speaking, we expect this extension applies to Shimura varieties for which the associated Shimura datum $(G,X)$ satisfies the condition that $G^{ad}$ is $\mathbb{Q}$-simple and has $\mathbb{Q}$-rank at least 2.

Recall that the \emph{coherent cohomological dimension} of a scheme $X$, denoted $\ccd(X)$, is defined as:
$$
\ccd(X) = \max\left\{ i \geq 0 \mid \exists \mathcal{F} \in \mathrm{Coh}(X),\ H^i(X, \mathcal{F}) \neq 0 \right\}.
$$

We fix a field $k$ that is algebraically closed and characteristic 0. When we talk about variety, we mean a reduced separated scheme of finite type over $k$, but it does not need to be irreducible. Let $A_{g,\Gamma}$ be the Siegel modular variety of genus $g\geq 2$ and arithmetic level $\Gamma\subset \Sp_{2g}(\mathbb{Q})$ (i.e., $\Gamma$ is commensurable with $\Sp_{2g}(\mathbb{Z})$ and contains a principal congruence subgroup $\Gamma(N)$ for some $N$). Our main result is the following.

\begin{thm}[Theorem \ref{ccd_bound}]\label{main_thm} 
Let $A_{g,\Gamma}$ be the Siegel modular variety of genus $g\geq2$ and level $\Gamma$, with dimension $d = \frac{1}{2}g(g+1)$. Then
$$
\ccd(A_{g,\Gamma}) \leq d - 2.
$$
\end{thm}

As a key application, we deduce that formal Siegel modular forms are automatically convergent, i.e., they are Siegel modular forms. This generalizes the earlier work of Bruinier and Raum \cite{bruinier2015kudla} and \cite{bruinier2024formal}.

It is conjectured that $\ccd (A_{g,\Gamma})=d-g$, but very few results are known when $g\ge 5$. In general, we conjecture the following.
\begin{coj}
    Let $(G,X)$ be a Shimura datum that $G^{ad}$ is simple, and $Sh_{K}(G,X)$ be the associated Shimura variety of level $K$. Then 
    $$
        \ccd(Sh_{K}(G,X))\leq \dim(Sh_{K}(G,X))-\mathrm{rank}_\mathbb{Q}G^{ad}.
    $$
\end{coj}

To establish Theorem \ref{main_thm}, we introduce and study the weak G2 and G3 properties with respect to a line bundle, which refine the classical notions of G2 and G3 subvarieties in the sense of Hironaka--Matsumura in \cite{hironaka1968formal}.

\begin{defn}
    Let $\mathcal{Z}$ be a locally Noetherian formal scheme. For any affine open subset $U$ of $\mathcal{Z}$, we define
    $$
        \mathcal{M}^0_\mathcal{Z}(U):=[\mathcal{O_Z}(U)]_0,
    $$
    where $[A]_0$ denotes the total ring of fractions of the ring $A$.
    Let $\mathcal{M_Z}$ be the sheafification of the presheaf $\mathcal{M}^0_{\mathcal{Z}}$, and set
    $$
        K(\mathcal{Z}):=H^0(\mathcal{Z},\mathcal{M_Z}).
    $$
\end{defn}

We are interested in the case $\mathcal{Z}=X_{/Y}$ is the completion of an irreducible variety $X$ along a closed subvariety $Y$. Let $\hat{i}:X_{/Y}\to X$ be the natural flat morphism, it induces a natural ring homomorphism 
$$
    \alpha_{X,Y}:K(X)\to K(X_{/Y}),
$$
where $K(X)$ is the field of rational functions of $X$. For any line bundle $\mathcal{L}$ on $X$ and $\hat{\mathcal{L}}:=\hat{i}^*\mathcal{L}$, let $K(X,\mathcal{L})$ (resp. $K(X_{/Y},\hat{\mathcal{L}})$) denote the subfield of $K(X)$ (resp. $K(X_{/Y})$) consisting of ratios of $\Gamma(X,\mathcal{L}^m)$ (resp. $\Gamma(X_{/Y},\hat{\mathcal{L}}^m)$) for $m\geq0$. Now we define the weak G2 and G3 properties.

\begin{defn}
Let $X$ be an irreducible projective variety, $Y \subset X$ a closed subvariety, and $\mathcal{L}$ a line bundle on $X$. Suppose $K(X_{/Y})$ is a field. Then:
\begin{itemize}
    \item $Y$ is \emph{weakly G2 with respect to $\mathcal{L}$} if $K(X) = K(X, \mathcal{L})$ and the natural map $\alpha_{X,Y}: K(X, \mathcal{L}) \to K(X_{/Y}, \hat{\mathcal{L}})$ is a finite field extension.
    \item $Y$ is \emph{weakly G3 with respect to $\mathcal{L}$} if $K(X) = K(X, \mathcal{L})=K(X_{/Y}, \hat{\mathcal{L}})$.
\end{itemize}
\end{defn}

Let $A^{\min}_{g,\Gamma}$ denote the minimal compactification, and let $A^{\Sigma}_{g,\Gamma}$ denote the toroidal compactification with respect to the cone decomposition $\Sigma$. Let $\omega$ denote the Hodge line bundle on $A_{g,\Gamma}^{\min}$ or its pullback on $A_{g,\Gamma}^{\Sigma}$. We use $D_\Gamma^{\min}$ and $D_{\Gamma}^{\Sigma}$ to denote the corresponding boundary of the compactifications. There is a set-theoretic description of $A_{g,\Gamma}^{\min}$ that
$$
    A_{g,\Gamma}^{\min}=\bigsqcup\limits_{j=0}^{g}\bigsqcup_{t\in I_j}A_{j,\Gamma_t}
$$
where $I_j$ is a finite set. We define the stratification of the boundary as
$$
    D_{\Gamma,l}^{\min}:=\bigsqcup\limits_{j=0}^{l}\bigsqcup_{t\in I_j}A_{j,\Gamma_t},
$$ and $D_{\Gamma,l}^{\Sigma}$ as its preimage in $A_{g,\Gamma}^{\Sigma}$.

Our main geometric result is the following.

\begin{thm} [Corollary \ref{weak_g3_l}]
The boundary $D_{\Gamma,l}^{\mathrm{min}}$ is weakly G3 with respect to the Hodge line bundle $\omega$ in $A_{g,\Gamma}^{\mathrm{min}}$ for $l\ge 1$.
\end{thm}
We can deduce Theorem \ref{main_thm} from the case $l=g-1$ of the geometric result above, which is proved in Theorem \ref{weak_G3}, and then we can prove for general $l\geq1$ by using Theorem \ref{main_thm}. This result also provides new examples of subvarieties that satisfy the $G3$ sense properties. To our knowledge, there are very few examples except subvarieties in homogeneous spaces. As a corollary, we can show that the pair $(A_{g,\Gamma}^{\mathrm{min}}, D_{\Gamma,l}^{\mathrm{min}})$ (resp. $(A_{g,\Gamma}^{\Sigma}, D_{\Gamma,l}^{\Sigma})$) satisfies the \emph{Grothendieck--Lefschetz condition}.

\begin{defn}
A pair $(X, Y)$ satisfies $\mathrm{Lef}(X, Y)$ if for every connected open neighborhood $U$ of $Y$ and every vector bundle $\mathcal{E}$ on $U$, the restriction
$$
H^0(U, \mathcal{E}) \to H^0(X_{/Y}, \hat{i}^* \mathcal{E})
$$
is an isomorphism, where $\hat{i}:X_{/Y}\to X$ is the natural morphism.
\end{defn}

\begin{thm}[Corollary \ref{lef_l}]
The pair $(A_{g,\Gamma}^{\mathrm{min}}, D_{\Gamma,l}^{\mathrm{min}})$ (resp. $(A_{g,\Gamma}^{\Sigma}, D_{\Gamma,l}^{\Sigma})$) satisfies $\mathrm{Lef}(A_{g,\Gamma}^{\mathrm{min}}, D_{\Gamma,l}^{\mathrm{min}})$ (resp. $\mathrm{Lef}(A_{g,\Gamma}^{\Sigma}, D_{\Gamma,l}^{\Sigma})$) for $l\ge1$. In particular,
$$
H^0(A_{g,\Gamma}^{\mathrm{min}}, \omega^k) \cong H^0(\hat{A}_{g,\Gamma}^{\mathrm{min}}, \hat{\omega}^k),
$$
where $\hat{A}_{g,\Gamma}^{\mathrm{min}}$ is the formal completion along the boundary stratum $D_{\Gamma,l}^{\mathrm{min}}$.
\end{thm}

These results are proved via a detailed analysis of the behavior of formal rational functions under proper morphisms, combined with a study of the cohomological dimension of complements of boundary strata. We also make essential use of a theorem of van der Geer on the minimal codimension of proper subvarieties in $A_g$, which ensures that every effective divisor intersects the boundary nontrivially. We also use the result on the congruence subgroup problem for $\Sp_{2g}(\mathbb{Q})$, which is a classical result of Bass--Milnor--Serre. For Shimura datum $(G,X)$ that $G^{ad}$ is simple and has $\mathbb{Q}$-rank at least 2, it is possible to have similar results, and then the same framework makes sense.

The paper is organized as follows:  
In section 2, we review the classical G2 and G3 properties and their behavior under morphisms.  
In section 3, we introduce the weak G2 and G3 properties relative to a line bundle.  
In section 4, we relate these to cohomological dimension and the Lefschetz condition.  
In section 5, we apply these results to Siegel modular varieties and prove the main theorems.

\vspace{0.3cm}

\noindent\textbf{Acknowledgement.} I would like to thank Liang Xiao for asking me to consider this problem. And I want to thank Jan Hendrik Bruinier, Kai-Wen Lan, Wenxuan Qi, Martin Raum, Yichao Tian, Zhiyu Tian, Peihang Wu, Deding Yang, Zhiwei Yun, Wei Zhang, and Yichao Zhang for helpful discussions.

\section{Properties G2 and G3}
 In this section, we review the notion of closed subvarieties with the G2 and G3 properties. We mainly follow Section 9 of the book \cite{badescu2004projective}.

\begin{defn}
    Let $\mathcal{Z}$ be a locally Noetherian formal scheme. For any affine open subset $U$ of $\mathcal{Z}$, we define
    $$
        \mathcal{M}^0_\mathcal{Z}(U):=[\mathcal{O_Z}(U)]_0,
    $$
    where $[A]_0$ denotes the total ring of fractions of the ring $A$.
    Let $\mathcal{M_Z}$ be the sheafification of the presheaf $\mathcal{M}^0_{\mathcal{Z}}$, and set
    $$
        K(\mathcal{Z}):=H^0(\mathcal{Z},\mathcal{M_Z}).
    $$
    The sheaf $\mathcal{M_Z}$ is called the \emph{sheaf of formal rational functions} and the ring $K(\mathcal{Z})$ is called the \emph{ring of formal rational functions}.
    In particular, if $Y$ is a closed subvariety of an algebraic variety $X$, then $X_{/Y}$ denotes the completion of $X$ along $Y$, and $K(X_{/Y})$ is called the \emph{ring of formal rational functions along $Y$}.
\end{defn}

The ring $K(X_{/Y})$ is not a field in general. However, the following proposition provides a sufficient condition for it to be a field.

\begin{prop}
    Let $X$ be an irreducible projective variety, and let $Y$ be a closed subvariety. If $X$ is normal at every point of $Y$ and $Y$ is connected, then $K(X_{/Y})$ is a field.
\end{prop}

This is \cite[Prop. 9.2]{badescu2004projective}. A proof can also be found in the original paper \cite{hironaka1968formal} by Hironaka and Matsumura. Under the assumption of the above proposition, there is a natural ring homomorphism
$$
    \alpha_{X,Y}:K(X)\to K(X_{/Y}).
$$
This leads to the following definition of properties G2 and G3.

\begin{defn}
    Assume that $X$ is an irreducible projective variety and $Y$ is a closed subvariety. 
    \begin{itemize}
        \item We say that $Y$ is G2 if $K(X_{/Y})$ is a field and $\alpha_{X,Y}$ makes it a finite field extension over $K(X)$.
        \item We say that $Y$ is G3 if $\alpha_{X,Y}$ is an isomorphism.
    \end{itemize}
\end{defn}

We give two baby examples of G2 and G3 subvarities.

\begin{ex}
    Let $X$ be the $n$-dimensional projective space $\mathbb{P}^n (n\geq2)$ over $k$ and $Y\cong\mathbb{P}^1$ is defined by $x_2=x_3=\cdots=x_n=0$, where $[x_0,\dots,x_n]$ is the homogeneous coordinate system of $\mathbb{P}^n$. We can compute the ring $K(X_{/Y})$ directly.

    Let $U_0$ and $U_1$ be the open subsets of $\mathbb{P}^n$ defined by $x_0 \neq 0$ and $x_1 \neq 0$. Then $Y_0=U_{0/(U_0\cap Y)}$ and $Y_1=U_{1/(U_1 \cap Y)}$ forms an open affine cover of $X_{/Y}$. By definition, we have
    $$
        K(X_{/Y})=\text{Ker} ([\mathcal{O}_{X_{/Y}}(Y_0)]_0 \oplus [\mathcal{O}_{X_{/Y}}(Y_1)]_0 \rightrightarrows [\mathcal{O}_{X_{/Y}}(Y_{0}\cap Y_1)]_0).
    $$
    The ring $[\mathcal{O}_{X_{/Y}}(Y_0)]_0$ is the total ring of $k[\frac{x_1}{x_0}]\llbracket \frac{x_2}{x_0},\dots ,\frac{x_n}{x_0}\rrbracket$, which consists of the homogeneous degree $0$ formal power series whose degrees of $x_i$ are bounded below except $x_0$. Elements in $[\mathcal{O}_{X_{/Y}}(Y_1)]_0$ have similar descriptions, and the maps are natural inclusions. Therefore, $K(X_{/Y})$ consists of the homogeneous degree $0$ formal power series with bounded degrees for all $x_i$. From this, we see that $K(X_{/Y})$ coincides with $K(X)$ and $Y$ is G3 in $X$.
\end{ex}

\begin{rmk}
    Hironaka and Matsumura proved that every positive-dimensional connected closed subvariety $Y$ of $\mathbb{P}^n (n\geq2)$ over an algebraic closed field $k$ is G3. See \cite[Thm. 9.14]{badescu2004projective} or \cite{hironaka1968formal}.
\end{rmk}

\begin{ex}
    Consider the group action of $G=\mathbb{Z}/(n+1)\mathbb{Z}$ on $X'=\mathbb{P}^n(n\ge3)$ over $\mathbb{C}$. The generator $\sigma$ of $G$ acts as
    $$
        \sigma([x_0,x_1,x_2,\dots,x_n])=[x_0,\zeta x_1,\zeta^2 x_2, \dots,\zeta^n x_n],
    $$
    where $\zeta$ is a primitive root of unit of order $n+1$. Let $X:=X'/G$ be the quotient scheme and $f:X'\to X$ be the canonical quotient morphism. Let $U'$ be the open subset of $X$ where $G$ acts freely and $U=f(U')$. It is clear that $f|_{U'}$ is a finite \'etale morphism of degree $n+1$, thus $f^*:K(X)\to K(X')$ is a finite field extension of degree $n+1$.
    
    Let $L$ be the line in $X'$ defined by 
    $$
        x_0=x_1,\ x_2=x_3,\ x_4=\cdots=x_n=0.
    $$
    and $Y=f(L)$. Then $L$ lies in $U'$ and we have $f|_{L}:L\to Y$ is an isomorphism. By  \cite[Lem. 9.19]{badescu2004projective} or compute directly, we have an isomorphism $\tilde{f}:X'_{/L}\cong X_{/Y}$ induced by $f$.

    Now consider the commutative diagram
    $$
    \begin{tikzcd}
        K(X)  \arrow[r,"\alpha_{X,Y}"] \arrow[d,"f^*"] 
        & K(X_{/Y})\arrow[d,"\tilde{f}^*"]\\
        K(X') \arrow[r,"\alpha_{X',L}"]
        & K(X'_{/L}).
    \end{tikzcd}
    $$
    We have $\alpha_{X',L}$ and $\tilde{f}^*$ are isomorphisms and $f^*$ is a finite field extension of degree $n+1$, thus $\alpha_{X',L}$ is a finite field extension of degree $n+1$ and $Y$ is G2 but not G3 in $X$.
\end{ex}

The following theorem describes the behavior of the ring of formal rational functions along a closed subvariety under proper surjective morphisms. 

\begin{thm}[{\cite[Thm. 9.11]{badescu2004projective}}] \label{formal_function_base_change}
    Let $f: X' \to X$ be a proper surjective morphism of irreducible varieties. Then for every closed subvariety $Y$ of $X$ there is a canonical isomorphism
    \[
    K(X'_{/f^{-1}(Y)}) \cong [K(X') \otimes_{K(X)} K(X_{/Y})]_0,
    \]
    where $[A]_0$ denotes the total ring of fractions of a commutative ring $A$.
\end{thm}

An immediate corollary is the following.
\begin{cor}\label{G2_under_morphism}
    Let $f:X'\to X$ be a proper surjective morphism of projective varieties, and $Y$ be a closed subvariety of $X$ such that both $K(X_{/Y})$ and $K(X'_{/f^{-1}(Y)})$ are fields. Assume that $f$ is generically finite and separable. Then $Y$ is G2 (resp. G3) in $X$ if and only if $f^{-1}(Y)$ is G2 (resp. G3) in $X'$. 
\end{cor}
\begin{proof}
    Since $f$ is generically finite and separable, $K(X')$ is a finite separable extension of $K(X)$. By the primitive element theorem, we can write $K(X') \cong K(X)[t]/(g(t))$ for some irreducible separable polynomial $g(t) \in K(X)[t]$. By Theorem \ref{formal_function_base_change}, we have an isomorphism
    $$
        K(X'_{/f^{-1}(Y)}) \cong [K(X') \otimes_{K(X)} K(X_{/Y})]_0\cong[K(X_{/Y})[t]/(g(t))]_0.
    $$
    Under our assumptions, the total ring of fractions of $K(X_{/Y})[t]/(g(t))$ is a field, which implies that $g(t)$ remains irreducible over $K(X_{/Y})$. Therefore, $K(X_{/Y})[t]/(g(t))$ is already a field, and we have
    $$
        K(X'_{/f^{-1}(Y)}) \cong K(X_{/Y})[t]/(g(t)) \cong K(X') \otimes_{K(X)} K(X_{/Y}).
    $$
    Now we can view $K(X')$ and $K(X_{/Y})$ as intermediate fields between $K(X)$ and $K(X'_{/f^{-1}(Y)})$. Moreover, $[K(X'_{/f^{-1}(Y)}):K(X_{/Y})] = \deg(g) = [K(X'):K(X)]$.
    $$
        \begin{tikzcd}
            & K(X'_{/f^{-1}(Y)}) & \\
            K(X') \arrow[ur, hook, "\alpha_{X',f^{-1}(Y)}"]& &K(X_{/Y}) \arrow[ul, hook,"f^*"]  \\
            &K(X) \arrow[ru, hook, "\alpha_{X,Y}"] \arrow[ul, hook, "f^*"] &
        \end{tikzcd}
    $$

    Now we prove the equivalence:

    ($\Rightarrow$) If $Y$ is G2, then $K(X_{/Y})$ is a finite extension of $K(X)$, so $K(X'_{/f^{-1}(Y)})$ is a finite extension of $K(X)$. In particular, $K(X'_{/f^{-1}(Y)})$ is finite over $K(X')$, hence $f^{-1}(Y)$ is G2. If $Y$ is G3, then it is obvious that $f^{-1}(Y)$ is G3.

    ($\Leftarrow$) If $f^{-1}(Y)$ is G2, then $K(X'_{/f^{-1}(Y)})$ is a finite extension of $K(X')$ hence a finite extension of $K(X)$. Thus $K(X_{/Y})$ is a finite extension of $K(X)$ and $Y$ is G2. If $f^{-1}(Y)$ is G3, then we have $[K(X'_{/f^{-1}(Y)}):K(X_{/Y})]=[K(X'):K(X)]=[K(X'_{/f^{-1}(Y)}):K(X)]$ hence $K(X_{/Y})\cong K(X)$ and $Y$ is G3.
\end{proof}

The following theorem of Hartshorne--Gieseker is an important tool in the study of formal rational functions. See [Corollary 9.20] \cite[Cor. 9.20]{badescu2004projective} and \cite[Thm. 4.3]{gieseker1977two}.
\begin{thm}[Hartshorne--Gieseker] \label{G3_section} 
    Let $Y$ be a connected closed subvariety of an irreducible normal projective variety $X$ which is G2 in $X$. Then there exists a finite surjective morphism $f: X' \to X$ of degree $[K(X_{/Y}) : K(X)]$ from an irreducible normal projective variety $X'$ such that the inclusion $Y \subset X$ lifts to an inclusion $i: Y \hookrightarrow X'$, $f$ is \'etale in a neighborhood of $i(Y)$, and $i(Y)$ is G3 in $X'$.
\end{thm}

There is a refinement stated by Badescu--Schneider.

\begin{thm}[Badescu--Schneider] \label{G3_section_refine} 
    Let $Y$ be a connected closed subvariety of a irreducible normal projective variety $X$ and $\zeta\in K(X_{/Y})$ be an algebraic element over $K(X)$. Then there exists a finite surjective morphism $f: X' \to X$ of degree $\deg(f)=\deg_{K(X)}(\zeta)$ from a irreducible normal projective variety $X'$ such that the inclusion $Y \subset X$ lifts to an inclusion $i: Y \hookrightarrow X'$, $f$ is \'etale in a neighborhood of $i(Y)$, and $K(X)(\zeta)= K(X')$.
\end{thm}

\begin{rmk}
    In Theorem \ref{G3_section}, the normal projective variety $X'$ is exactly the (relative) normalization of $X$ in $\spec K(X_{/Y})$ via the natural morphism $\spec K(X_{/Y})\to \spec K(X)\to X$ given by $\alpha_{X,Y}$. Similarily in Theorem \ref{G3_section_refine} $X'$ is the normalization of $X$ in $\spec K(X)(\zeta)$.
\end{rmk}

Combining Theorem \ref{formal_function_base_change} we have the following.

\begin{prop}\label{normalization_base_change}
Let $g: X_1 \to X_2$ be a proper morphism of normal projective varieties, and let $Y_2$ be a G2 closed subvariety of $X_2$ such that its preimage $Y_1 = g^{-1}(Y_2)$ is G2 in $X_1$. 
\begin{enumerate}[label=(\arabic*)]
    \item There exists a unique morphism $g':X'_1\to X'_2$ fitting into the following commutative diagram:
    $$
        \begin{tikzcd}
            \spec K(X_{2/Y_2}) \arrow[r] \arrow[d]&
            \spec K(X_{1/Y_1}) \arrow[d]\\
            X'_2 \arrow[r,"g'"] \arrow[d,"f_2"]&
            X'_1 \arrow[d,"f_1"]\\
            X_2 \arrow[r,"g"]&
            X_1,
        \end{tikzcd}
    $$
    where $f_i:X'_i\to X_i$ is given by the relative normalization of $X_i$ in $\spec K(X_{i/Y_i})$. 
    
    \item Assume that $g$ is generically finite and generically \'etale. Then for any open subscheme $U_2 \subset X_2$ such that the restriction $g|_{U_1}: U_1 \to U_2$ is finite \'etale, where $U_1 = g^{-1}(U_2)$, we have the following Cartesian diagram:
     $$
        \begin{tikzcd}
            f_2^{-1}(U_1) \arrow[r,"g'"] \arrow[d,"f_2"]
            & f_1^{-1}(U_2) \arrow[d,"f_1"] \\
            U_1 \arrow[r,"g"]
            & U_2.
        \end{tikzcd}
     $$
\end{enumerate}

\end{prop}

\begin{proof}
    The first part follows from the universal property of relative normalization (See \cite[\href{https://stacks.math.columbia.edu/tag/0BAK}{Tag 0BAK}]{stacks-project}), since we have the following commutative diagram of ring maps by Theorem \ref{formal_function_base_change}
    $$
        \begin{tikzcd}
            K(X_{2/Y_2})
            &K(X_{1/Y_1})\otimes_{K(X_1)}K(X_2) \arrow[l]
            &K(X_{1/Y_1}) \arrow[l]\\
            &K(X_2) \arrow[u] \arrow[ul]
            &K(X_1) \arrow[l] \arrow[u].
        \end{tikzcd}
    $$
    For the second part, both $K(X_{1/Y_1})$ and $K(X_2)$ are finite extensions of $K(X_1)$, thus we have 
    $$ 
        K(X_{2/Y_2})=K(X_{1/Y_1})\otimes_{K(X_1)}K(X_2).
    $$
    By \cite[\href{https://stacks.math.columbia.edu/tag/03GE}{Tag 03GE}]{stacks-project}, taking the integral closure commutes with \'etale base change, so the considered diagram is Cartesian.
\end{proof}

\section{Weak G2 and G3 properies with respect to a line bundle}

In this section, we define weak versions of G2 and G3 properties to focus on the behavior of line bundles.

\begin{defn}
    Let $X$ be an irreducible projective variety and $\mathcal{L}$ a line bundle on $X$. Fix an embedding $\mathcal{L}\hookrightarrow K_X$ where $K_X$ is the constant sheaf of rational functions. We define 
    $$
        K(X,\mathcal L)=\left\{ \frac{s_0}{s_1}\in K(X) \middle|\ s_0,s_1\in H^0(X,\mathcal{L}^m) \textrm{ for some } m\geq 0 \right\}.
    $$
    It is obvious that $K(X,\mathcal{L})$ is independent of the choice of the embedding $\mathcal{L}\hookrightarrow K_X$.
\end{defn}

\begin{prop}
    Let $X$ be an irreducible projective variety and $\mathcal{L}$ a line bundle on $X$. 
    \begin{enumerate}[label=(\arabic*)]
        \item For any injective morphism $\mathcal{L}\hookrightarrow \mathcal{L'}$, we have $K(X,\mathcal{L})\subset K(X,\mathcal{L}')$.
        \item For any integer $n\geq0$, we have $K(X,\mathcal{L})= K(X,\mathcal{L}^n)$.
        \item For any ample line bundle $\mathcal{L}$, we have $K(X,\mathcal{L})=K(X)$.
    \end{enumerate}
\end{prop}

\begin{proof}
    The part $(1)$ is trivial. For $(2)$, it is trivial that $K(X,\mathcal{L}^n)\subset K(X,\mathcal{L})$. Conversely, if $s_0/s_1\in K(\mathcal{L})$ with $s_0,s_1\in H^0(X,\mathcal{L}^m)$, then $s_0s_1^{n-1},s_1^{n}\in H^0(X,\mathcal{L}^{nm})$ and therefore $s_0/s_1=s_0s_1^{n-1}/s_1^n$ lies in $K(X,\mathcal{L}^n)$.
    
    For $(3)$, using $(2)$, we may assume that $\mathcal{L}$ is very ample. This gives a closed immersion $j:X\hookrightarrow \mathbb{P}^N$ over $k$ such that $j^*\mathcal{O}_{\mathbb{P^N}}(1)=\mathcal{L}$, and the coordinates are given by $s_0,\dots,s_N\in H^0(X,\mathcal{L})$. Then $K(X)$ can be generated by ratios of $s_0,\dots,s_N$ and therefore is equal to $K(X,\mathcal{L})$. 
\end{proof}

\begin{defn}
    Let $X$ be an irreducible projective variety and $\mathcal{L}$ a line bundle on $X$. Let $Y$ be a closed subvariety such that $K(X_{/Y})$ is a field, and let $\hat{i}:X_{/Y}\to X$ be the natural flat morphism. Fix an embedding $\mathcal{L}\hookrightarrow K_X$, it induces an embedding $\hat{\mathcal{L}}:=\hat{i}^* \mathcal{L}\hookrightarrow \mathcal{M}_{X_{/Y}}$. Now we define 
    $$
        K(X_{/Y},\hat{\mathcal{L}})=\left\{ \frac{s_0}{s_1}\in K(X_{/Y}) \middle|\ s_0,s_1\in H^0(X_{/Y},\hat{\mathcal{L}}^m) \textrm{ for some } m\geq 0 \right\}.
    $$ It is obvious that $K(X_{/Y},\hat{\mathcal{L}})$ is also independent of the choice of embedding. Since $\hat{i}:X_{/Y}\to X$ is flat, we have $\alpha_{X,Y}:K(X)\to K(X_{/Y})$ sending $K(X,\mathcal{L})$ into $K(X_{/Y},\hat{\mathcal{L}})$.
\end{defn}

\begin{defn}
    Assume that $X$ is an irreducible projective variety and $Y$ is a closed subvariety. Let $\mathcal{L}$ be a line bundle on $X$. 
    \begin{enumerate}
        \item We say that $Y$ is \emph{weakly G2 with respect to $\mathcal{L}$} if $K(X_{/Y})$ is a field, $K(X)=K(X,\mathcal{L})$ and $\alpha_{X,Y}$ makes $K(X_{/Y},\mathcal{L})$ a finite field extension over $K(X,\mathcal{L})$.
        \item We say that $Y$ is \emph{weakly G3 with respect to $\mathcal{L}$} if $K(X_{/Y})$ is a field and $K(X)=K(X,\mathcal{L})=K(X_{/Y},\mathcal{L})$ via $\alpha_{X,Y}$.
    \end{enumerate}
\end{defn}

The following propositions are weak versions of Corollary \ref{G2_under_morphism}.
\begin{prop} \label{weak_G2_galois}
    Let $G$ be a finite group that acts faithfully on an irreducible projective variety $X'$ and 
    $$f:X'\longrightarrow X=X'/G$$ 
    be the natural quotient morphism. Let $Y$ be a closed subvariety of $X$ such that both $K(X_{/Y})$ and $K(X'_{/f^{-1}(Y)})$ are fields.  Let $\mathcal{L}$ be a line bundle on $X$ and $\mathcal{L'}=f^*\mathcal{L}$, such that $K(X,\mathcal{L})=K(X)$ and $K(X',\mathcal{L'})=K(X')$. Then $Y$ is weakly G2 (resp. G3) in $X$ with respect to $\mathcal{L}$ if and only if $f^{-1}(Y)$ is G2 (resp. G3) in $X'$ with respect to $\mathcal{L}'$. 
\end{prop}

\begin{proof}
    It is obvious that $G$ is the Galois group of $K(X')$ over $K(X)$. The action of $G$ on $X'$ induces a natural group action on $H^0(X'_{/f^{-1}(Y)}, \hat{\mathcal{L}'})$ and we have
    \begin{align*}
          H^0(X'_{/f^{-1}(Y)}, \hat{\mathcal{L}'})^{G}=H^0(X_{/Y},\hat{\mathcal{L}}).
    \end{align*}

    Therefore, we have $K(X'_{/f^{-1}(Y)},\hat{\mathcal{L}'})^G=K(X_{/Y},\hat{\mathcal{L}})$. As $G$ acts faithfully on $K(X')$, it also acts faithfully on $K(X'_{/f^{-1}(Y)},\hat{\mathcal{L}'})$. So $K(X'_{/f^{-1}(Y)},\hat{\mathcal{L}'})$ is a Galois extension of degree $|G|$ over $K(X_{/Y},\hat{\mathcal{L}})$.
    Similarly to the proof of Corollary \ref{G2_under_morphism}, we have $K(X'_{/f^{-1}(Y)})) \cong K(X') \otimes_{K(X)} K(X_{/Y})$ and thus
    $$
        K(X') \otimes_{K(X)} K(X_{/Y},\hat{\mathcal{L}})\subset K(X'_{/f^{-1}(Y)},\hat{\mathcal{L}'}).
    $$
    As both sides are degree $|G|$ extensions over $K(X_{/Y},\hat{\mathcal{L}})$ they are equal. The rest are the same as in Corollary \ref{G2_under_morphism}. 
\end{proof}

\begin{prop} \label{weak_G2_birational}
    Let $f:X'\to X$ be a proper birational morphism of projective varieties, and $Y$ be a closed subvariety of $X$ such that both $K(X_{/Y})$ and $K(X'_{/f^{-1}(Y)})$ are fields. Assume further that $X$ is normal. Let $\mathcal{L}$ be a line bundle on $X$ and $\mathcal{L}'=f^*\mathcal{L}$. Then $Y$ is weakly G2 (resp. G3) in $X$ with respect to $\mathcal{L}$ if and only if $f^{-1}(Y)$ is G2 (resp. G3) in $X'$ with respect to $\mathcal{L}'$.  
\end{prop}

\begin{proof}
    It is sufficient to show that $H^0(X'_{/f^{-1}(Y)},\hat{\mathcal{L}}'^{\otimes n})=H^0(X_{/Y},\hat{\mathcal{L}}^{\otimes n})$. We show that $f_*\mathcal{O}_{X_{/f^{-1}(Y)}'}=\mathcal{O}_{X_{/Y}}$ then the above equality holds by the projection formula. By Zariski's main theorem we have $f_*\mathcal{O}_{X'}=\mathcal{O}_X$, thus we have $f_*\mathcal{O}_{X_{/f^{-1}(Y)}'}=\mathcal{O}_{X_{/Y}}$ by the flatness of $X_{/Y}\to X$.
\end{proof}

We also have a weak version of Proposition \ref{normalization_base_change}.

\begin{prop} \label{weak_G2_base change}
Let $g: X_2 \to X_1$ be a proper morphism of normal projective varieties, and let $\mathcal{L}_1$ be a line bundle on $X_1$ and $\mathcal{L}_2=g^*\mathcal{L}_1$. Let $Y_1$ be a closed subvariety in $X_1$ and $Y_2=f^{-1}(Y_1)$. Assume that $Y_1$ is weakly G2 with respect to $\mathcal{L}_1$ and $Y_2$ is weakly G2 with respect to $\mathcal{L}_2$ in $X_2$. 
    \begin{enumerate}[label=(\arabic*)]     
    \item There exists a unique morphism $g':X'_1\to X'_2$ fitting into the following commutative diagram:     
    $$         
        \begin{tikzcd}             
            \spec K(X_{2/Y_2},\hat{\mathcal{L}}_2) \arrow[r] \arrow[d]&             
            \spec K(X_{1/Y_1},\hat{\mathcal{L}}_1) \arrow[d]\\             
            X'_2 \arrow[r,"g'"] \arrow[d,"f_2"]&             
            X'_1 \arrow[d,"f_1"]\\             
            X_2 \arrow[r,"g"]&             
            X_1       ,  
        \end{tikzcd}     
    $$     
    
    where $f_i:X'_i\to X_i$ is given by repeatedly using of Theorem \ref{G3_section_refine}.
    
    \item Assume that there is a finite group $G$ acting faithfully on $X_2$ and $g:X_2\to X_1=X_2/{G}$ is the natural morphism onto the quotion scheme. Then for any open subscheme $U_2 \subset X_2$ such that the restriction $g|_{U_1}: U_1 \to U_2$ is finite \'etale (hence Galois), where $U_1 = g^{-1}(U_2)$, we have the following Cartesian diagram.      
    $$         
        \begin{tikzcd}             
        f_2^{-1}(U_1) \arrow[r,"g'"] \arrow[d,"f_2"]
        & f_1^{-1}(U_2) \arrow[d,"f_1"] \\             
        U_1 \arrow[r,"g"]             
        & U_2       .  
        \end{tikzcd}      
    $$ 
    \end{enumerate} 
\end{prop}

\begin{proof}
    The proof is the same as Proposition \ref{normalization_base_change} using the commutative diagram
    $$
        \begin{tikzcd}
            K(X_{2/Y_2},\hat{\mathcal{L}}_2)
            &K(X_{1/Y_1},\hat{\mathcal{L}}_1)\otimes_{K(X_1)}K(X_2) \arrow[l]
            &K(X_{1/Y_1}) \arrow[l]\\
            &K(X_2) \arrow[u] \arrow[ul]
            &K(X_1) \arrow[l] \arrow[u]
        \end{tikzcd}
    $$ and Proposition \ref{weak_G2_galois}.
\end{proof}

\section{Cohomological dimension and Grothendieck--Lefschetz condition}

In this section, we review the notion of (coherent) cohomological dimension and Grothendieck--Lefschetz condition, and how they are related to the weak G2 and G3 properties. 

\begin{defn}
    The \emph{coherent cohomological dimension} of a scheme $X$ is defined as
    $$
        \ccd(X)=\text{max} \left\{i\geq 0 \middle|\ \text{there exists } \mathcal{F}\in \mathrm{Coh}(X) \text{  such that } H^i(X,\mathcal{F})\neq 0 \right\}.
    $$
    
\end{defn}

The following lemma relates $\ccd(X)$ to vanishings of the cohomology groups of twists of negative powers of an ample line bundle.

\begin{lemma} \label{ccd_ample}
    Let $\mathcal{O}(1)$ be an ample line bundle on $X$. Write $\mathcal{O}(m)$ for its $m$-th power. Then $\ccd(X)\leq r$ if and only if there exist $M\in \mathbb{Z}_{\ge0}$ and a line bundle $\mathcal{L}$ that $H^i(X,\mathcal{L}(-m))=0$ for all $i>r$ and $m\geq M$, where $\mathcal{L}(-m):=\mathcal{L}\otimes\mathcal{O}(-m)$.
\end{lemma}

\begin{proof}
     It suffices to prove the if part. We prove $ \ccd(X)\leq i$ by reverse induction on $i\geq r$. When $i>\dim X$, this is trivial. Suppose that $\ccd(X)\leq i+1$ and $i\geq r$, we prove that $\ccd(X)\leq i$. For any coherent sheaf $\mathcal{F}$ on $X$, we have that $ \mathcal{H}(m):=\mathcal{F}\otimes\mathcal{L}^{-1}(m)$ is generated by global sections for sufficiently large $m$ (and $m\ge M$). Then we have an exact sequence
    $$
        0\longrightarrow \mathcal{G}\longrightarrow \mathcal{O}_X^{\oplus N}\longrightarrow \mathcal{H}(m) \longrightarrow0
    $$ with a coherent sheaf $\mathcal{G}$, and then an exact sequence
    $$
        0\longrightarrow \mathcal{G}\otimes\mathcal{L}(-m)\longrightarrow \mathcal{L}(-m)^{\oplus N}\longrightarrow \mathcal{F} \longrightarrow0.
    $$
    Taking cohomology groups gives
    $$
        0=H^i(X,\mathcal{L}(-m)^{\oplus N})\longrightarrow H^i(X,\mathcal{F})\longrightarrow H^{i+1}(X,\mathcal{G}\otimes\mathcal{L}(-m))=0
    $$ is exact, hence $H^i(X,\mathcal{F})=0$ and then $\ccd(X)\leq i$.
\end{proof}

\begin{defn}
    Let $Y$ be a closed subvariety of an irreducible projective variety $X$. We say that the pair $(X, Y)$ \emph{satisfies the Grothendieck--Lefschetz condition $\mathrm{Lef}(X, Y)$} if for every connected open subset $U$ of $X$ containing $Y$ and every vector bundle $\mathcal{E}$ on $U$ the natural map $H^0(U, \mathcal{E}) \to H^0(X_{/Y}, \hat{i}^* \mathcal{E})$ is an isomorphism, where $\hat{i}: X_{/Y} \to U$ is the canonical morphism.
\end{defn}

\begin{rmk} \label{Lef_diff}
    We will review some results in the book \cite{hartshorne2006ample} and \cite{badescu2004projective}, in which the definitions of Grothendieck--Lefschetz condition are slightly different. The above definition is the same as \cite{badescu2004projective}, and the two definitions are the same when $X$ satisfies Serre's condition $S_2$ and $Y$ intersects nontrivally with every effective divisor of $X$ (using the algbraic Hartogs' theorem).
\end{rmk}

The above concepts are highly related, especially in the case that $X$ is smooth. The following proposition is exactly the same as \cite[Chap. \uppercase\expandafter{\romannumeral 4}, Prop. 1.1]{hartshorne2006ample}.

\begin{prop} \label{ccd=lef}
    Let $Y$ be a closed subvarity of an irreducible smooth projective variety $X$ of dimension $d\geq 2$ and $\hat{i}:X_{/Y}\to X$ denote the natural morphism. Then the following are equivalent:
    \begin{enumerate}[label=(\arabic*)]
        \item $\ccd(X\backslash Y)\leq d-2$,
        \item $\mathrm{Lef}(X,Y)$ and $Y$ intersects every effective divisor on $X$.
        \item Let $\mathcal{O}(1)$ be an ample line bundle on $X$. There exists an $M\in \mathbb{Z}_{\ge0}$ that the canonical map $H^0(X_{/Y},\hat{i}^*\mathcal{O}(m))\cong H^0(X,\mathcal{O}(m))$ is an isomorphism for $m\geq M$.
    \end{enumerate}
\end{prop}

\begin{proof}
    $(1)\Rightarrow (2)$: It is obvious that there exists no proper effective divisor in $X\backslash Y$. We show that $\mathrm{Lef}(X,Y)$ is satisfied. Let $U$ be an open subset in $X$ containing $Y$, and $\mathcal{E}_{U}$ be a vector bundle on $U$. Let $\mathcal{E}_{X}$ be a coherent sheaf on $X$ that extends $\mathcal{E}_{U}$, and let $\mathcal{F}_{X}=\mathcal{H}om_{\mathcal{O}_X}(\mathcal{E}_{X},\omega_X)$ be the sheaf of homomorphism from $\mathcal{E}_X$ to the dualizing sheaf $\omega_X$ of $X$. Then we have $\mathcal{F}_{U}:=\mathcal{F}_X|_{U}=\mathcal{H}om_{\mathcal{O}_U}(\mathcal{E}_U,\omega_X|_{U})$ is a vector bundle on $U$.

    Now, by Serre duality, we have
    $$
        H^0(X,\mathcal{E}_X)\cong (\ext^d(\mathcal{E}_X,\omega_X))'=(H^d(X,\mathcal{F}_X))',
    $$ where $'$ denotes the dual vector space. By Harshorne's formal duality (See \cite[Thm 3.3]{hartshorne2006ample}, in this case it is the following simple application of Serre duality), we have
    \begin{align*}
        H^0(X_{/Y},\hat{i}^*\mathcal{E}_X)
        &\cong \mathrm{lim}H^0(X,\mathcal{E}_X\otimes\mathcal{O}_X/\mathcal{I}_Y^n)\\
        &\cong (\mathrm{colim} {\ext}^d(\mathcal{E}_X\otimes\mathcal{O}_X/\mathcal{I}_Y^n,\omega_X))' \\
        &\cong(\mathrm{colim} {\ext}^d(\mathcal{O}_X/\mathcal{I}_{Y}^n,\mathcal{F}_X))'\\
        &\cong(H_{Y}^d(X,\mathcal{F}_X))',
    \end{align*} where $H^i_{Y}(X,\mathcal{F}_X)$ is the local cohomology of $X$ with support in $Y$.
    
    For local cohomology, we have the exact sequence
    $$
        H^{d-1}(X\backslash Y,\mathcal{F}_X) \to
        H^{d}_{Y}(X,\mathcal{F}_X)\to 
        H^d(X,\mathcal{F}_X)\to 
        H^d(X\backslash Y, \mathcal{F}_X),
    $$ and by condition $\ccd(X\backslash Y)\leq d-2$ we have 
    $$
        H^{d}_{Y}(X,\mathcal{F}_X)\cong H^d(X,\mathcal{F}_X)
    $$
    Combining the duality results above, we have 
    $$
        H^0(X_{/Y},\hat{i}^*\mathcal{E}_X)\cong H^0(X, \mathcal{E}_X)
    $$
    As $\mathcal{E}_U$ is locally free, we have an injective sheaf homomorphism $\mathcal{E}_U\hookrightarrow \hat{i}_*\hat{i}^*\mathcal{E}_U$ on $U$, which implies injections
    $$
        H^0(X, \mathcal{E}_X)\hookrightarrow H^0(U,\mathcal{E}_U)\hookrightarrow H^0(U_{/Y},\hat{i}^*\mathcal{E}_U).
    $$
    Thus we have $H^0(U,\mathcal{E}_U)\cong H^0(U_{/Y},\hat{i}^*\mathcal{E}_U)$ hence $\mathrm{Lef}(X,Y)$ holds.
    
    $(2)\Rightarrow (3)$: Apply $\mathrm{Lef}(X,Y)$ to the case $U=X$ and $\mathcal{E}=\mathcal{O}(m)$.
    
    $(3)\Rightarrow (1)$: Since $\mathcal{O}(1)$ is an ample line bundle on $X$, it suffices to prove that $H^i(X\backslash Y,\omega_X(-m))=0$ when $i>d-2$ and $m$ are large enough according to the Lemma \ref{ccd_ample}. Again we use the condition and the duality theorems above, we have
    $$
        H^{d}_{Y}(X,\omega_{X}(-m))\cong H^d(X,\omega_{X}(-m)).
    $$ By Serre vanishing, we have
    $$
    H^{d-1}(X,\omega_{X}(-m))\cong H^1(X,\mathcal{O}(m))=0
    $$ for $m$ large enough.
    Now, the expected vanishing follows from the exact sequence
    \begin{align*}
        0\to H^{d-1}(X\backslash Y,\omega_X(-m)) \to H^{d}_Y(X,\omega_X(-m)) \xrightarrow{\sim}
        H^d(X,\omega_X(-m))\to H^d(X\backslash Y, \omega_X(-m))\to 0.
    \end{align*}

\end{proof}

The following theorem of Hartshorne and Speiser states a relation between the G3 property and the Grothendieck--Lefschetz condition. See \cite[Thm. 10.7]{badescu2004projective} or \cite[Chap.\uppercase\expandafter{\romannumeral 5}, Prop. 2.1]{hartshorne2006ample}.

\begin{thm}[Hartshorne--Speiser]
    Let $Y$ be a closed subvariety of an irreducible projective variety $X$ (of dimension $d\geq 2$) that locally satisfies Serre's condition $S_2$ (e.g. when $X$ is normal). Assume that $Y$ is G3 in $X$ and $Y$ intersects nontrivially with every effective divisor of $X$. Then $(X,Y)$ satisfies the condition $\mathrm{Lef}(X,Y)$.
\end{thm}

We need a slightly different version that replaces G3 with weakly G3 with respect to an ample line bundle.

\begin{prop} \label{g3_to_lef}
    Let $Y$ be a closed subvariety of a projective variety $X$ (of dimension $d\geq 2$) that locally satisfies Serre's condition $S_2$ (e.g. when $X$ is normal). Assume that $Y$ is weakly G3 in $X$ with respect to an ample line bundle $\mathcal{L}$ and $Y$ intersects nontrivially with every effective divisor of $X$. Then $(X,Y)$ satisfies the condition $\mathrm{Lef}(X,Y)$.
\end{prop}

\begin{proof}
    The proof is almost the same as the proof of the above theorem. For an integer $m\geq 0$, we fix an embedding $\mathcal{L}^m\hookrightarrow K_X$, and then induce an embedding $\hat{\mathcal{L}}^m:=\hat{i}^*\mathcal{L}^m\to \mathcal{M}_{X_{/Y}}$ where $\hat{i}:X_{/Y}\to X$ is the natural morphism. We have the commutative diagram
    $$
        \begin{tikzcd}
             H^0(X,\mathcal{L}^m) \arrow[r,hook] \arrow[d]
            &K(X,\mathcal{L})=K(X)  \arrow[d,"\alpha_{X,Y}",hook]\\
            H^0(X_{/Y},\hat{\mathcal{L}}^m) \arrow[r,hook]
            &K(X_{/Y},\hat{\mathcal{L}}).
        \end{tikzcd}
    $$
    Since $Y$ is weakly G3 with respect to $\mathcal{L}$, the vertical right arrow is an isomorphism. For a section $s\in K(X_{/Y},\hat{\mathcal{L}})$, it has no poles at any point $y\in Y$, as $\mathcal{O}_{X_{/Y},y}\cap K(X)=\mathcal{O}_{X,y}$. So, there exists an open subset $V$ of $X$ that contains $Y$, such that $s\in H^0(V,\mathcal{L}^m)$. Then $s$ can be extended to $X$ by the algebraic Hartogs' Theorem. 

    Now we prove that $(X,Y)$ satisfies $\mathrm{Lef}(X,Y)$. For any open subset $U$ of $X$ containing $Y$, we have $H^0(X,\mathcal{L}^m)\cong H^0(U,\mathcal{L}_U^m)\cong H^0(X_{/Y},\hat{\mathcal{L}}^m)$ by the above argument. For any vector bundle $\mathcal{E}$ on $U$, we denote $\mathcal{E}^{\vee}$ as its dual sheaf. Then by coherence of $\mathcal{E}^{\vee}$ and ampleness of $\mathcal{L}_U$, there exists a sufficiently large $m$ that $\mathcal{E}^{\vee}\otimes \mathcal{L}_U^m$ is generated by global sections and further sits in the exact sequence
    $$
        \mathcal{O}_U^{\oplus k}\longrightarrow \mathcal{O}_U^{\oplus l}\longrightarrow \mathcal{E}^{\vee}\otimes \mathcal{L}_U^m\longrightarrow 0 .
    $$
    Tensoring $\mathcal{L}^{-m}$ and dualizing, we get an exact sequence
    $$
        0\longrightarrow\mathcal{E}\longrightarrow (\mathcal{L}_U^m)^{\oplus l}\longrightarrow (\mathcal{L}_U^m)^{\oplus k}.
    $$
    Pulling back to $X_{/Y}$, we also have 
    $$
        0\longrightarrow \hat{\mathcal{E}}\longrightarrow (\hat{\mathcal{L}}^m)^{\oplus l}\longrightarrow (\hat{\mathcal{L}}^m)^{\oplus k}.
    $$
    Therefore, we get the commutative diagram with exact rows
    $$
        \begin{tikzcd}
            0 \arrow[r]
            & H^0(U,\mathcal{E}) \arrow[r] \arrow[d]
            & H^0(U,(\mathcal{L}_U^m)^{\oplus l}) \arrow[r] \arrow[d]
            & H^0(U,(\mathcal{L}_U^m)^{\oplus k}) \arrow[d] \\
            0 \arrow[r]
            & H^0(X_{/Y},\hat{\mathcal{E}}) \arrow[r]
            & H^0(X_{/Y},(\hat{\mathcal{L}}^m_U)^{\oplus l}) \arrow[r]
            & H^0(X_{/Y},(\hat{\mathcal{L}}^m_U)^{\oplus k}) .
        \end{tikzcd}
    $$
    The two vertical maps on the right are isomorphisms, thus the left vertical is also an isomorphism. So the condition $\mathrm{Lef}(X,Y)$ holds.
\end{proof}

The following is the weakly G3 version of the converse direction of \cite[Prop. 2.1]{hartshorne2006ample}. It is much simpler because the weak G3 property is much weaker than the G3 property, and as we mentioned in Remark \ref{Lef_diff}, the Grothendieck--Lefschetz condition $\mathrm{Lef}(X,Y)$ in the book $\cite{hartshorne2006ample}$ is slightly different from ours. 

\begin{prop}\label{lef_to_g3}
    Let $Y$ be a closed subvariety of an irreducible projective variety $X$ (of dimension $d\geq 2$) that $K(X_{/Y})$ is a field and $\mathcal{L}$ be a line bundle on $X$ that $K(X,\mathcal{L})=K(X)$. If the pair $(X,Y)$ satisfies the condition $\mathrm{Lef}(X,Y)$, then $Y$ is weakly G3 in $X$ with respect to $\mathcal{L}$.
\end{prop}

\begin{proof}
    The condition $\mathrm{Lef}(X,Y)$ implies that $H^0(X,\mathcal{L}^m)\cong H^0(X_{/Y},\hat{\mathcal{L}}^m)$ and then $K(X)=K(X,\mathcal{L})=K(X_{/Y},\hat{\mathcal{L}})$.
\end{proof}

We end this section by stating an analogue of the following result of Hartshorne--Speiser. It is a combination of the previous analogues.

\begin{thm}[Hartshorne--Speiser]
    Let $Y$ be a closed subvariety of an irreducible smooth projective variety $X$ of dimension $d \geq 2$. Then the following are equivalent:
    \begin{enumerate}
        \item $Y$ is G2 in $X$ and $\ccd(X\backslash Y)\leq d-2$.
        \item $Y$ is G3 in $X$ and $Y$ intersects nontrivially with every effective divisor on $X$.
    \end{enumerate}
\end{thm}
\begin{proof}
    See \cite[Rmk. 11.24]{badescu2004projective} or \cite[Chap. \uppercase\expandafter{\romannumeral 5}, Cor. 2.2]{hartshorne2006ample}. 
    
\end{proof}

\begin{prop}\label{ccd=g3}
    Let $Y$ be a closed subvariety of an irreducible smooth projective variety $X$ of dimension $d \geq 2$, and $\mathcal{L}$ an ample line bundle on $X$. Then the following are equivalent:
    \begin{enumerate}
        \item $\ccd(X\backslash Y)\leq d-2$.
        \item $Y$ is weakly G3 in $X$ with respect to $\mathcal{L}$ and $Y$ intersects nontrivially with every effective divisor on $X$.
    \end{enumerate}
\end{prop}

\begin{proof}
    Combine Proposition \ref{ccd=lef} and \ref{lef_to_g3}.
\end{proof}

\section{Main results}

In this section, we first show that the boundaries of the minimal compactifications of Siegel modular varieties satisfy the weak G3 properties with respect to the Hodge line bundle $\omega$, and then apply the results in the previous section to show our main theorems.

We let $A_g$ denote the Siegel modular variety (over $k$) of genus $g\ge2$ and level $1$ and let $A_{g,\Gamma}$ be the Siegel modular variety of genus $g\ge2$ and arithmetic level $\Gamma\subset \Sp_{2g}(\mathbb{Q})$ (i.e. $\Gamma$ is commensurable with $\Sp_{2g}(\mathbb{Z})$ and contains a principal congruence subgroup $\Gamma(N)$ for some $N$). Let $A_g^{\min}$ and $A^{\min}_{g,\Gamma}$ denote the minimal compactifications, and let $A_g^{\Sigma}$ and $A^{\Sigma}_{g,\Gamma}$ denote the toroidal compactifications with respect to the cone decomposition $\Sigma$. We use $D^{\min}$, $D_\Gamma^{\min}$, $D^{\Sigma}$ and $D_{\Gamma}^{\Sigma}$ to denote the corresponding boundary of the compactifications. It is well known that $A_{g,\Gamma}^{\min}$ and $A_{g,\Gamma}^{\Sigma}$ are normal and the boundaries are connected.

Let $\omega$ be the Hodge line bundle on $A_g$ and $\omega_{\Gamma}$ be the Hodge line bundle on $A_{g,\Gamma}$. We use $\omega^{\min}$ and $\omega^{\can}$ (resp. $\omega_{\Gamma}^{\min}$ and $\omega_{\Gamma}^{\can}$) to denote the certain extension and canonical extension of the Hodge line bundle on minimal and toroidal compactifications of $A_g$ (resp. $A_{g,\Gamma}$). We still use $\omega$ to denote it if there is no doubt which variety we are talking about.

Bruinier and Raum defined the notion of formal Siegel modular forms of weight $k$ and cogenus $l$ in Section $3$ in \cite{bruinier2024formal}. The complex vector space of the formal Siegel modular forms of weight $k$, cogenus $l$ and level $\Gamma$ is denoted as $\mathrm{FM}^{(g,l)}_{k}(\Gamma)$. We shall not recall their definition here, since they have proved that $\mathrm{FM}_k^{(g,1)}(\Gamma)$ can be realized as the global section of $\omega^k$ on the formal completion of $A_{g,\Gamma}^{\min}$ along the boundary. More precisely, let $\hat{A}^{\min}_{g,\Gamma}=(A^{\min}_{g,\Gamma})_{/D^{\min}_{\Gamma}}$ denote the formal completion and $\hat{\omega}$ be the pullback of $\omega$ on $\hat{A}^{\min}_{g,\Gamma}$.

\begin{thm}[Bruinier-Raum]
    The canonical map 
    $$
        H^0(\hat{A}^{\min}_{g,\Gamma}, \hat{\omega}^k)\longrightarrow
        \mathrm{FM}^{(g,1)}_{k}(\Gamma)
    $$ is an isomorphism.
\end{thm}

Now we prove that $D_{\Gamma}^{\min}$ is weakly G2 in $A_{g,\Gamma}^{\min}$ with respect to $\omega$.

\begin{lemma}\label{tr_deg}
    Let $A=\bigoplus_{m}A_m$ be a graded $k$-algebra, which is an integral domain, and let $K(A)$ be its quotient field. Suppose that there is a polynomial $P(z)\in \mathbb{Q}[z]$ of degree $n$, such that 
    $$
        \dim_kA_m\leq P(m)
    $$ for all $m$ sufficiently large. Then
    \begin{enumerate} [label=(\arabic*)]
        \item The transcendence degree $\trdeg K(A)/k\leq n+1$
        \item If the $\trdeg K(A)/k=n+1$ then $K(A)$ is a finitely generated field extension over $k$.
    \end{enumerate}
\end{lemma}
    
\begin{proof}
    The proof is not hard; for example, see \cite[Lem. 6.3]{hartshorne1968cohomological}. 
\end{proof}

We hope to give a proof of the dimension bound of formal sections by using algebraic geometric language in a forthcoming paper (joint with Peihang Wu, to deal with general cases), as we have certain analogues of the results about the existence of slope bound and theta decomposition. However, in this paper, we simply apply the result on the dimension bound of formal Siegel modular forms of level $1$.

\begin{prop}
    The boundary $D^{\min}$ is weakly G2 in $A_g^{\min}$ with respect to $\omega$ when $g\geq 2$.
\end{prop}

\begin{proof}
    The ampleness of $\omega$ implies $K(\hat{A}_g^{\min},\hat{\omega})=K(A_g^{\min})$, where $\hat{A}_g^{\min}=(A_{g}^{\min})_{/D^{\min}}$ is the formal completion along the boundary and $\hat{\omega}$ is the pullback of $\omega$.
    Let $A$ be the graded ring
    $$
        A=\bigoplus\limits_{m\geq0} H^0(\hat{A}_g^{\min},\hat{\omega}^m),
    $$ 
    and by \cite[Thm. 3.11]{bruinier2015kudla} we deduce that $A$ satisfies the condition in Lemma \ref{tr_deg} by a degree $d=g(g+1)/2$ polynomial $P(z)$. Assume that $\trdeg K(\hat{A}_g^{\min},\hat{\omega})/k=r$, and we choose a basis of algebraically independent elements $\xi_1,\dots,\xi_r\in K(\hat{A}_g^{\min},\hat{\omega})$. After replacing $\xi_i$ with suitable powers, we can assume further that $\xi_i=s_i/s_0$ for some $s_0,\dots,s_r\in H^0(\hat{A}_{g}^{\min},\hat{\omega}^{m_0})$. It is not difficult to check that $s_0,\dots,s_r$ are algebraically independent in $K(A)$ over $k$. So we can conclude that $\trdeg K(A)\geq r+1$ hence $\trdeg K(\hat{A}_g^{\min},\hat{\omega})/k\leq d$. Since $K(\hat{A}_g^{\min},\hat{\omega})$ contains $K(A_g^{\min})$, we have
    $$\trdeg K(\hat{A}_g^{\min},\hat{\omega})=\trdeg K(A_g^{\min})=d.$$
    
    Now we show that $K(\hat{A}_g^{\min},\hat{\omega})$ is a finitely generated field extension over $k$. Therefore, it is a finitely generated field extension over $K(A_g^{\min})$, and hence a finite extension over $K(A_g^{\min})$. We consider the homomorphism of graded rings
    $$
        A\longrightarrow K(\hat{A}_g^{\min},\hat{\omega})[T]
    $$ which sends $s\in H^0(\hat{A}_g^{\min},\hat{\omega}^m)$ to $(s/s_0^m)T^m$. We can write 
    $$
        A=H^0(\hat{A}_g^{\min},\bigoplus\limits_{m\geq 0}\hat{\omega}^m),
    $$ and 
    $$
        K(\hat{A}_g^{\min},\hat{\omega})[T] \subset K(\hat{A}_g^{\min})[T] =H^0(\hat{A}_g^{\min}, \mathcal{M}_{\hat{A}_g^{\min}}[T]).
    $$
    Since $A_g^{\min}$ is normal, we see that $\mathcal{O}_{\hat{A}_g^{\min}}$ is integrally closed in $\mathcal{M}_{\hat{A}_g^{\min}}$ and hence $\mathcal{O}_{\hat{A}_g^{\min}}[T]$ is integrally closed in $\mathcal{M}_{\hat{A}_g^{\min}}[T]$. As $\bigoplus_{m\geq 0}\hat{\omega}^m$ is locally isomorphic to $\mathcal{O}_{\hat{A}_g^{\min}}[T]$, we also see that $\bigoplus_{m\geq 0}\hat{\omega}^m$ is integrally closed in $\mathcal{M}_{\hat{A}_g^{\min}}[T]$. Hence $A$ is integrally closed in $K(\hat{A}_g^{\min},\hat{\omega})[T]$. Since both quotient fields have the same transcendence degree $d+1$, we must have
    $$
        K(A)=K(\hat{A}_g^{\min},\hat{\omega})(T).
    $$
    The second part of Lemma \ref{tr_deg} tells us that $K(A)$ is a finitely generated field extension over $k$, hence $K(\hat{A}_g^{\min},\hat{\omega})$ is also a finitely generated field extension over $k$. So $K(\hat{A}_g^{\min},\hat{\omega})$ is finite over $K(A_g^{\min})$ thus $D^{\min}$ is weakly G2 with respect to $\omega$.
\end{proof}

\begin{cor}
    The boundary $D^{\min}_{\Gamma}$ (resp. $D^{\Sigma}_{\Gamma}$) is weakly G2 with respect to $\omega$ in $A_{g,\Gamma}^{\min}$ (resp. $A_{g,\Gamma}^{\Sigma}$).
\end{cor}

\begin{proof}
    By Proposition \ref{weak_G2_birational}, it suffices to show the claim in the minimal compactification cases. We first show that for any two arithmetic level groups $\Gamma_1\subset \Gamma_2$,  the boundary $D^{\min}_{\Gamma_1}$ is weakly G2 with respect to $\omega$ in $A_{g,\Gamma_1}^{\min}$ if and only if $D^{\min}_{\Gamma_2}$ is weakly G2 with respect to $\omega$ in $A_{g,\Gamma_2}^{\min}$. Choose an arithmetic subgroup $\Gamma_0\subset\Gamma_1\subset \Gamma_2$ that is normal in $\Gamma_2$. Then we have $A_{g,\Gamma_i}^{\min}=(A_{g,\Gamma_0}^{\min})/G_i$ where $G_i=\Gamma_i/\Gamma_0$, and the canonical maps between the minimal compactifications are exactly the natural quotient maps.
    $$
        \begin{tikzcd}
            D_{\Gamma_0}^{\min}  \arrow[d] \arrow[r,hook]
            &A_{g,\Gamma_0}^{\min}  \arrow[d] \\
            D_{\Gamma_1}^{\min}  \arrow[d] \arrow[r,hook]
            &A_{g,\Gamma_1}^{\min}  \arrow[d] \\
            D_{\Gamma_2}^{\min}   \arrow[r,hook]
            &A_{g,\Gamma_2}^{\min}.   \\
        \end{tikzcd}
    $$
    Applying Proposition \ref{weak_G2_galois}, we see that $D^{\min}_{\Gamma_1}$ is weakly G2 with respect to $\omega$ in $A_{g,\Gamma_1}^{\min}$ if and only if $D^{\min}_{\Gamma_0}$ is weakly G2 with respect to $\omega$ in $A_{g,\Gamma_0}^{\min}$, which is further equivalent to $D^{\min}_{\Gamma_2}$ is weakly G2 with respect to $\omega$ in $A_{g,\Gamma_2}^{\min}$. The claim of this corollary follows from the fact that $\Gamma \cap\Gamma(1)$ is an arithmetic group.
    
\end{proof}

To prove that the boundaries are weakly G3 with respect to $\omega$, we need one more essential input.

\begin{prop}\label{etale_tower}
    Assume that $\Gamma\subset \Sp_{2g}(\mathbb{Q})$ is a neat level and $B\to A_{g,\Gamma}$ is a finite \'etale morphism. If $B$ is connected, then $B\cong A_{g,\Gamma_0}$ for some arithmetic level $\Gamma_0\subset\Gamma$.
\end{prop}
\begin{proof}
    Fix a closed point $x \in A_{g,\Gamma}$. Its \'etale fundamental group is given by
    $$
        \pi_1(A_{g,\Gamma}, x) \cong \widehat{\Gamma} = \varprojlim_{\substack{[\Gamma : \Gamma'] < \infty \\ \Gamma' \triangleleft \Gamma}} \Gamma/\Gamma',
    $$
    where the limit runs over all finite-index normal subgroups of $\Gamma$. By the Galois correspondence for finite \'etale covers, the connected cover $B$ corresponds to a finite-index open subgroup $H$ of $\widehat{\Gamma}$. Such a subgroup $H$ pulls back to a finite-index subgroup $\Gamma_0 \subset \Gamma$.

    By the main theorem of \cite{bass1967solution}, $\Gamma_0$ contains a principal congruence subgroup $\Gamma(N)$ for some $N$ and is therefore an arithmetic level subgroup itself. Under the Galois correspondence, $\Gamma_0$ corresponds to the Siegel modular variety $A_{g,\Gamma_0}$, and hence $B \cong A_{g,\Gamma_0}$.
\end{proof}

Now we can prove that the boundaries are actually weakly G3 with respect to $\omega$.

\begin{thm} \label{weak_G3}
    The boundary $D^{\min}_{\Gamma}$ (resp. $D^{\Sigma}_{\Gamma}$) is weakly G3 with respect to $\omega$ in $A_{g,\Gamma}^{\min}$ (resp. $A_{g,\Gamma}^{\Sigma}$).
\end{thm}

\begin{proof}
    By Proposition \ref{weak_G2_galois}, it is sufficient to prove the theorem in the minimal compactification cases. We first deal with the case where $\Gamma$ is neat and $A_{g,\Gamma}$ is smooth, and the general case follows by the same argument as in the previous corollary.
    Applying Theorem \ref{G3_section_refine} repeatedly, there exists a normal projective variety $B_{\Gamma}^{\min}$, which is the relative normalization of $A_{g,\Gamma}^{\min}$ in $K(\hat{A}_{g,\Gamma}^{\min},\hat{\omega})$, and a canonical finite morphism $f_{\Gamma}:B_{\Gamma}^{\min}\to A_{g,\Gamma}^{\min}$ that admits a lift of the embedding $i:D_{\Gamma}^{\min}\hookrightarrow A_{g,\Gamma}^{\min}$.
    $$
    \begin{tikzcd}
        &&D_{\Gamma}^{\min} \arrow[d,hook,"i"] \arrow[ld,hook,"i'"swap]\\
        \spec K(\hat{A}_{g,\Gamma}^{\min},\hat{\omega}) \arrow[r]   &B_{\Gamma}^{\min} \arrow[r,"f_{\Gamma}"]
        &A_{g,\Gamma}^{\min}.
    \end{tikzcd}
    $$
    Moreover $f_{\Gamma}$ is \'etale on an open neighbourhood of $i'(D_{\Gamma}^{\min})$. Let $B_{\Gamma}=f_{\Gamma}^{-1}(A_{g,\Gamma})$ and $\Delta'_{\Gamma}\subset B_{\Gamma}$ be the ramification locus of $f_{\Gamma}$. Therefore we know that $\Delta'_{\Gamma}$ is contained in a codimension $1$ closed subvaritey. We denote $\Delta_{\Gamma}:=f_{\Gamma}(\Delta'_{\Gamma})\subset A_{g,\Gamma}$ and will prove that $\Delta_{\Gamma}=\emptyset$.
    $$
    \begin{tikzcd}
        \Delta'_{\Gamma} \arrow[r] \arrow[d,hook]
        &\Delta_{\Gamma} \arrow[d,hook]   \\
        B_{\Gamma} \arrow[r,"f_{\Gamma}"]
        & A_{g,\Gamma}.
    \end{tikzcd}
    $$
    Our method is to prove that $\Delta_{\Gamma}$ is stable under Hecke actions. That means, for any $x \in \Sp_{2g}(\mathbb{Q})$ and arithmetic level $\Gamma'\subset \Gamma\cap x^{-1}\Gamma x$, consider the following diagram
    $$
        \begin{tikzcd}            
            & A_{g,\Gamma '} \arrow[ld,"\pi_{\Gamma',\Gamma}"swap] \arrow[r,"{[x]}"]
            & A_{g, x\Gamma' x^{-1}} \arrow[rd, "\pi_{x\Gamma' x^{-1},\Gamma}"]
            & \\
            A_{g,\Gamma}
            &&& A_{g,\Gamma}
        \end{tikzcd}
    $$ 
    where $\pi_{\Gamma',\Gamma}$ is the natural map given by the inclusion of levels and $[x]$ is the isomorphism given by $x$, we will prove that $\mathcal{H}_x(\Delta_{\Gamma}):=\pi_{x\Gamma' x^{-1},\Gamma}\circ[x](\pi_{\Gamma,\Gamma'}^{-1}(\Delta_{\Gamma}))\subset \Delta_{\Gamma}$.

    We claim that for any two arithmetic level groups $\Gamma_1\subset \Gamma_2$ we have $\pi_{\Gamma_1,\Gamma_2}^{-1}(\Delta_{\Gamma_2})=\Delta_{\Gamma_1}$, then 
    $$
        \mathcal{H}_x(\Delta_{\Gamma})=\pi_{x\Gamma' x^{-1},\Gamma}\circ[x](\Delta_{\Gamma'})=\pi_{x\Gamma' x^{-1},\Gamma}(\Delta_{x\Gamma'x^{-1}})\subset \Delta_{\Gamma}.
    $$ 
    Without loss of generality, we can further assume $\Gamma_1$ is normal in $\Gamma_2$, and we denote $G=\Gamma_2/\Gamma_1$ which acts freely on $A_{g,\Gamma_1}$. Applying Theorem \ref{weak_G2_base change}, we have a Cartesian square
    $$
        \begin{tikzcd}
            B_{\Gamma_1} \arrow[d,"f_{\Gamma_1}"swap] \arrow[r,"{\tilde{\pi}_{\Gamma_1,\Gamma_2}}"]
            & B_{\Gamma_2} \arrow[d,"f_{\Gamma_2}"] \\
            A_{g,\Gamma_1} \arrow[r,"\pi_{\Gamma_1,\Gamma_2}"]
            &A_{g,\Gamma_2}.
        \end{tikzcd}
    $$
    By \cite[\href{https://stacks.math.columbia.edu/tag/0476}{Tag 0476}]{stacks-project} and the flatness of $\pi_{\Gamma_1,\Gamma_2}$, we have $\tilde{\pi}_{\Gamma_1,\Gamma_2}^{-1}(B_{\Gamma_2}\backslash\Delta_{\Gamma_2}')=B_{\Gamma_1}\backslash\Delta_{\Gamma_1}'$, and hence we have $\pi_{\Gamma_1,\Gamma_2}^{-1}(\Delta_{\Gamma_2})=\Delta_{\Gamma_1}$.
    
    Now we can conclude that $\Delta_{\Gamma}$ is empty because the Hecke orbit of every closed point $z$ in $A_{g,\Gamma}$ is Zariski dense but $\Delta_{\Gamma}$ is contained in a closed subvariety of codimension $1$. Therefore, $f_{\Gamma}:B_{\Gamma}\to A_{g,\Gamma}$ is finite \'etale and then coincides $\pi_{\Gamma_0,\Gamma}:A_{g,\Gamma_0}\to A_{g,\Gamma}$ for some $\Gamma_0\subset \Gamma$ by Proposition \ref{etale_tower}. By definition $B_{\Gamma}^{\min}$ is the relative normalization of $A^{\min}_{g,\Gamma}$ in $K(B_{\Gamma})=K(A_{g,\Gamma_0})$, thus $B_{\Gamma}^{\min}=A_{g,\Gamma_0}^{\min}$. So we have $B_{\Gamma}^{\min}\backslash B_{\Gamma}$ is connected thus it has to be $i'(D_{\Gamma}^{\min})$. As $f_{\Gamma}:B^{\min}_{\Gamma}\to A_{g,\Gamma}^{\min}$ is \'etale on $B_{\Gamma}$ and an open neighborhood of $B_{\Gamma}^{\min}\backslash B_{\Gamma}$, it is \'etale on the whole variety $B_{\Gamma}^{\min}$. Since $f_{\Gamma}$ is an isomorphism on $f^{-1}_{\Gamma}(D_{\Gamma}^{\min})$, we find that $f_{\Gamma}$ is an isomorphism and hence $D_{\Gamma}^{\min}$ is weakly G3 in $A_{g,\Gamma}^{\min}$ with respect to $\omega$.
\end{proof}

Now we recall the result on the maximal dimension of the proper subvarieties in $A_g$. A bound is given by van der Geer in \cite[Cor. 2.7]{van1999cycles}.

\begin{thm}[van der Geer]\label{complete_dim_bound}
    A proper subvariety of $A_g$ has codimension at least $g$.
\end{thm}
\begin{rmk}
    The explicit bound of the minimal codimension of proper subvariety of $A_g$ is given by Grushevsky, Mondello, Manni, and Tsimerman in \cite{grushevsky2024compact}
\end{rmk}

Now we can apply the propositions in the previous section to get the results on cohomological dimension and Grothendieck--Lefschetz conditions, and in particular on automatic convergence.

\begin{thm} \label{lef_g-1}
    The pair $(A_{g,\Gamma}^{\min},D^{\min}_{\Gamma})$ (resp. $(A_{g,\Gamma}^{\Sigma},D^{\Sigma}_{\Gamma})$) satisfies the Grothendieck--Lefschetz condition $\mathrm{Lef}(A_{g,\Gamma}^{\min},D^{\min}_{\Gamma})$ (resp. $\mathrm{Lef}(A_{g,\Gamma}^{\Sigma},D^{\Sigma}_{\Gamma})$). In particular, We have
    $$
        H^0(A_{g,\Gamma}^{\min},\omega^k)\cong H^0(\hat{A}_{g,\Gamma}^{\min},\hat{\omega}^k).
    $$
\end{thm}
\begin{proof}
    By Theorem \ref{complete_dim_bound}, every effective divisor intersects nontrivially with the boundary, thus the condition $\mathrm{Lef}(A_{g,\Gamma}^{\min},D^{\min}_{\Gamma})$ follows from Proposition \ref{lef_to_g3} and Theorem \ref{weak_G3}. For the pair $(A_{g,\Gamma}^{\Sigma},D^{\Sigma}_{\Gamma})$, we need the weak G3 property with respect to the ample line bundle $\omega^k(-D^{\Sigma}_{\Gamma})$ where $k$ is an integer large enough. This follows from the fact that the composition of the inculsion maps
    $$
        K(A_{g,\Gamma}^{\Sigma})\subset K(\hat{A}_{g,\Gamma}^{\Sigma},\hat{\omega}^{k}(-D^{\Sigma}_{\Gamma}))\subset K(\hat{A}_{g,\Gamma}^{\Sigma},\hat{\omega}^{k})=K(A_{g,\Gamma}^{\Sigma})
    $$ is an isomorphism, by Theorem \ref{weak_G3}.
\end{proof}

\begin{thm}\label{ccd_bound}
    Let $d=\dim A_{g,\Gamma}=g(g+1)/2$, then the cohomological dimension $\ccd(A_{g,\Gamma})\leq d-2$.
\end{thm}
\begin{proof}
    Firstly, we assume that there exists a toroidal compactification that $A_{g,\Gamma}^{\Sigma}$ is smooth. Since the boundary $D^{\Sigma}_{\Gamma}$ is weakly G3 in $A_{g,\Gamma}^{\Sigma}$ with respect to the ample line bundle $\omega^k(-D^{\Sigma}_{\Gamma})$, then $\ccd(A_{g,\Gamma})\leq d-2$ by Proposition \ref{ccd=g3}. 
    
    For general cases, choose a small enough normal arithmetic subgroup $\Gamma'\subset\Gamma$, we have proved $\ccd(A_{g,\Gamma'})\leq d-2$. Denoting $G=\Gamma/\Gamma'$, we have $A_{g,\Gamma}$ is the coarse moduli of the tame quotient stack $[A_{g,\Gamma'}/G]$, thus we have the spectral sequence
    $$  
        E_2^{p,q}=H^{p}(G,H^{q}(A_{g,\Gamma'},\omega_{\Gamma'}^k)) \Rightarrow H^{p+q}([A_{g,\Gamma'}/G],i^*\omega_{\Gamma}^k)=H^{p+q}(A_{g,\Gamma},\omega_{\Gamma}^k)
    $$
    where $i:[A_{g,\Gamma'}/G]\to A_{g,\Gamma}$ is the map from a stack to its coarse moduli, and the last equality is given by the fact that $i_*$ is exact and $i_*i^*\omega_\Gamma^k=\omega_{\Gamma}^k\otimes i_*\mathcal{O}_{[A_{g,\Gamma'}/G]}=\omega_{\Gamma}^k$. Since $G$ is finite, the classical result is that $H^p(G,M)$ is $|G|$-torsion for $p>0$, hence vanish if $M$ is uniquely $|G|$-divisible. So we have $E^{p,q}_2=0$ if $p>0$, thus
    $$
        H^{q}(A_{g,\Gamma'},\omega_{\Gamma'}^k)^G=H^{q}(A_{g,\Gamma},\omega_{\Gamma}^k)
    $$
    for any integer $k$. By Lemma \ref{ccd_ample} and the ampleness of $\omega_\Gamma$ we can conclude that $\ccd(A_{g,\Gamma})\leq d-2$.
\end{proof}

We also consider the stratification of the boundary to deal with the properties of the smaller strata. The classical result on the description of minimal compactification $A_{g,\Gamma}^{\min}$ is that
$$
    A_{g,\Gamma}^{\min}=\bigsqcup\limits_{j=0}^{g}\bigsqcup_{t\in I_j}A_{j,\Gamma_t}
$$
where $I_j$ is a finite set. We define the stratification of the boundary as
$$
    D_{\Gamma,l}^{\min}:=\bigsqcup\limits_{j=0}^{l}\bigsqcup_{t\in I_j}A_{j,\Gamma_t},
$$
and then $D_{\Gamma,g-1}^{\min}=D_{\Gamma}^{\min}$. For toroidal compactification $A_{g,\Gamma}^{\Sigma}$, we denote 
$$
    D_{\Gamma,l}^{\Sigma}:=\pi^{-1}(D^{\min}_{\Gamma,l})
$$
where $\pi:A_{g,\Gamma}^{\Sigma}\to A_{g,\Gamma}^{\min}$ is the canonical morphism. Let $F_{\Gamma,l}^{\Sigma}=D_{\Gamma,l}^{\Sigma} \backslash D_{\Gamma,l-1}^{\Sigma}$ denote the preimage of $\bigsqcup_{t\in I_l}A_{l,\Gamma_t}$

\begin{prop}
    Assume that $A_{g,\Gamma}^{\Sigma}$ is smooth and $D_{\Gamma}^{\Sigma}$ is a Cartier divisor. Then $\ccd(A_{g,\Gamma}^{\Sigma}\backslash D_{\Gamma,l}^{\Sigma})\leq d-2$ for $l\geq 1$, where $d=\dim A_{g,\Gamma}^{\Sigma}=g(g+1)/2$.
\end{prop}
\begin{proof}
    By definition, the morphism
    $$\pi:F_{\Gamma,l}^{\Sigma} \to \bigsqcup_{t\in I_j}A_{l,\Gamma_t}
    $$
    restricted from $\pi$ has relative dimension $d-l(l+1)/2-1$. Applying Theorem \ref{ccd_bound} we have
    $$
        \ccd(F_{\Gamma,j}^{\Sigma})=\ccd(\bigsqcup_{t\in I_j}A_{j,\Gamma_t})+d-j(j+1)/2-1\leq d-3
    $$ if $j>l$
    For any ample line bundle $\mathcal{L}$ on $A_{g,\Gamma}^{\Sigma}$, we have the excision spectral sequence
    $$
        E_1^{p,q}=
        H^{p+q}_{F_{\Gamma,g-p}^{\Sigma}}
        (A_{g,\Gamma}^{\Sigma}\backslash D_{\Gamma,g-p-1}^{\Sigma}, \mathcal{L})
        \Rightarrow 
        H^{p+q}(A_{g,\Gamma}^{\Sigma}\backslash D_{\Gamma,l}^{\Sigma}, \mathcal{L}), \quad0\leq p\leq l-1,
    $$
    with respect to the stratum
    $$
        D_{\Gamma,l+1}^{\Sigma} \backslash D_{\Gamma,l}^{\Sigma}\subset
        D_{\Gamma,l+2}^{\Sigma} \backslash D_{\Gamma,l}^{\Sigma}\subset\cdots
        \subset D_{\Gamma,g-1}^{\Sigma} \backslash D_{\Gamma,l}^{\Sigma}\subset
        A_{g,\Gamma}^{\Sigma} \backslash D_{\Gamma,l}^{\Sigma}.
    $$
    The cohomology group with support can be computed by the spectral sequence 
    $$
        E_2^{r,s}=H^r(F_{\Gamma,j}^{\Sigma}, \mathcal{H}^s_{F_{\Gamma,j}^{\Sigma}}(A_{g,\Gamma}^{\Sigma}\backslash D_{\Gamma,j-1}^{\Sigma}, \mathcal{L})) \Rightarrow H^{r+s}_{F_{\Gamma,j}^{\Sigma}}
        (A_{g,\Gamma}^{\Sigma}\backslash D_{\Gamma,j-1}^{\Sigma}, \mathcal{L}),
    $$
    where $\mathcal{H}^s_{F_{\Gamma,j}^{\Sigma}}(A_{g,\Gamma}^{\Sigma}\backslash D_{\Gamma,j-1}^{\Sigma}, \mathcal{L})$ is the sheaf of cohomology with support, see \cite[\href{https://stacks.math.columbia.edu/tag/0A39}{Tag 0A39}]{stacks-project} for a precise definition. As $D_{\Gamma}^{\Sigma}$ is a Cartier divisor, we have $\mathcal{H}^s_{F_{\Gamma,j}^{\Sigma}}(A_{g,\Gamma}^{\Sigma}\backslash D_{\Gamma,j-1}^{\Sigma}, \mathcal{L})=0$ unless $s=1$. Since $\ccd(F_{\Gamma,j}^{\Sigma})\leq g(g+1)/2-3$, we have $E_2^{r,s}=0$ if $r> g(g+1)/2-3$ and $s\neq 1$. Thus in the first spectral sequence, $E_1^{p,q}=0$ if $p+q> d-2$. Therefore, $H^{i}(A_{g,\Gamma}^{\Sigma}\backslash D_{\Gamma,l}^{\Sigma}, \mathcal{L})=0$ if $i>d-2$, hence $\ccd(A_{g,\Gamma}^{\Sigma}\backslash D_{\Gamma,l}^{\Sigma})\leq d-2$ by Lemma \ref{ccd_ample}.
\end{proof}

\begin{cor} \label{weak_g3_l}
    The boundary stratum $D^{\min}_{\Gamma,l}$ is weakly G3 in $A_{g,\Gamma}^{\min}$ with respect to $\omega$ for $l\ge 1$. 
\end{cor}
\begin{proof}
    Choose a normal arithmetic subgroup $\Gamma'\subset\Gamma$ small enough such that there exists a smooth toroidal compactification $A_{g,\Gamma'}^{\Sigma}$. Then by the previous proposition and Proposition \ref{ccd=lef}, we have that every effective divisor in $A_{g,\Gamma'}^{\Sigma}$ intersects nontrivially with $D^{\Sigma}_{\Gamma',l}$ and $\mathrm{Lef}(A_{g,\Gamma'}^{\Sigma},D^{\Sigma}_{\Gamma',l})$. Thus, every effective divisor in $A_{g,\Gamma}^{\min}$ intersects nontrivially with $D^{\min}_{\Gamma,l}$, and by Proposition \ref{lef_to_g3}, $D^{\Sigma}_{\Gamma',l}$ is weakly G3 in $A_{g,\Gamma'}^{\Sigma}$ with respect to $\omega$. By Theorem of formal functions, we have 
    $$
       \pi_* \mathcal{O}_{\hat{A}_{g,\Gamma'}^{\Sigma}}= \mathcal{O}_{\hat{A}_{g,\Gamma'}^{\min}}
    $$
    where $\hat{A}_{g,\Gamma'}^{\min}=(A_{g,\Gamma'}^{\min})_{/D^{\min}_{\Gamma',l}}$ and $\hat{A}_{g,\Gamma'}^{\Sigma}=(A_{g,\Gamma'}^{\Sigma})_{/D^{\Sigma}_{\Gamma',l}}$ are formal completions. Hence $\pi_*\hat{\omega}^{k}=\hat{\omega}^k$ and then
    $$
       K(A_{g,\Gamma}^{\min})=K(A_{g,\Gamma}^{\Sigma})=K(\hat{A}_{g,\Gamma'}^{\Sigma},\hat{\omega})=K(\hat{A}_{g,\Gamma'}^{\min},\hat{\omega}).
    $$
    So $D^{\min}_{\Gamma',l}$ is weakly G3 in $A_{g,\Gamma'}^{\min}$ with respect to $\omega$, and then $D^{\min}_{\Gamma,l}$ is weakly G3 in $A_{g,\Gamma}^{\min}$ with respect to $\omega$ by Proposition \ref{weak_G2_galois}. 
     
\end{proof}

\begin{cor} \label{lef_l}
    The pair $(A_{g,\Gamma}^{\min},D^{\min}_{\Gamma,l})$ (resp. $(A_{g,\Gamma}^{\Sigma},D^{\Sigma}_{\Gamma,l})$) satisfies the Grothendieck--Lefschetz condition $\mathrm{Lef}(A_{g,\Gamma}^{\min},D^{\min}_{\Gamma,l})$ (resp. $\mathrm{Lef}(A_{g,\Gamma}^{\Sigma},D^{\Sigma}_{\Gamma,l})$) for $l\geq 1$. In particular, We have
    $$
        H^0(A_{g,\Gamma}^{\min},\omega^k)\cong H^0(\hat{A}_{g,\Gamma}^{\min},\hat{\omega}^k)
    $$
    where $\hat{A}_{g,\Gamma}^{\min}=(A_{g,\Gamma}^{\min})_{/D^{\min}_{\Gamma,l}}$ is the formal completion and $\hat{\omega}$ is the pullback of $\omega$.
\end{cor}
\begin{proof}
    The minimal compactification case follows from Corollary \ref{weak_g3_l} and Proposition \ref{g3_to_lef}.
    For the toroidal compactification case, we have that $D^{\Sigma}_{\Gamma,l}$ is weakly G3 in $A_{g,\Gamma}^{\Sigma}$ with respect to $\omega$ by Proposition \ref{weak_G2_birational}, and then weakly G3 with respect to $\omega^k(-D^{\Sigma}_{\Gamma})$ by the same argument as in Theorem \ref{lef_g-1}. Therefore, $\mathrm{Lef}(A_{g,\Gamma}^{\Sigma},D^{\Sigma}_{\Gamma,l})$ follows from Proposition \ref{g3_to_lef}.
\end{proof}

\bibliographystyle{alpha}

\begin{thebibliography}{GMMT24}

\bibitem[Bad04]{badescu2004projective}
Lucian~Silvestru Badescu.
\newblock {\em Projective geometry and formal geometry}, volume~65.
\newblock Springer Science \& Business Media, 2004.

\bibitem[BMS67]{bass1967solution}
Hyman Bass, John Milnor, and Jean-Pierre Serre.
\newblock Solution of the congruence subgroup problem for $\mathrm{SL}_n$ ($n\geq 3$) and $\mathrm{Sp}_{2n}$ ($n\geq 2$).
\newblock {\em Publications Mathématiques de l'IHÉS}, 33:59--137, 1967.

\bibitem[BR15]{bruinier2015kudla}
Jan Bruinier and Martin Raum.
\newblock Kudla’s modularity conjecture and formal fourier--jacobi series.
\newblock In {\em Forum of Mathematics, Pi}, volume~3, page~e7. Cambridge University Press, 2015.

\bibitem[BR24]{bruinier2024formal}
Jan Bruinier and Martin Raum.
\newblock Formal siegel modular forms for arithmetic subgroups.
\newblock {\em Transactions of the American Mathematical Society, Series B}, 11(41):1394--1434, 2024.

\bibitem[Gie77]{gieseker1977two}
David Gieseker.
\newblock On two theorems of griffiths about embeddings with ample normal bundle.
\newblock {\em American Journal of Mathematics}, 99(6):1137--1150, 1977.

\bibitem[GMMT24]{grushevsky2024compact}
Samuel Grushevsky, Gabriele Mondello, Riccardo~Salvati Manni, and Jacob Tsimerman.
\newblock Compact subvarieties of the moduli space of complex abelian varieties.
\newblock {\em arXiv preprint arXiv:2404.06009}, 2024.

\bibitem[Har68]{hartshorne1968cohomological}
Robin Hartshorne.
\newblock Cohomological dimension of algebraic varieties.
\newblock {\em Annals of Mathematics}, 88(3):403--450, 1968.

\bibitem[Har06]{hartshorne2006ample}
Robin Hartshorne.
\newblock {\em Ample subvarieties of algebraic varieties}, volume 156.
\newblock Springer, 2006.

\bibitem[HM68]{hironaka1968formal}
Heisuke Hironaka and Hideyuki Matsumura.
\newblock Formal functions and formal embeddings.
\newblock {\em Journal of the Mathematical Society of Japan}, 20(1-2):52--82, 1968.

\bibitem[Kud97]{kudla1997algebraic}
Stephen~S Kudla.
\newblock Algebraic cycles on shimura varieties of orthogonal type.
\newblock 1997.

\bibitem[Liu12]{liu2012arithmetic}
Yifeng Liu.
\newblock Arithmetic theta lifting and l-derivatives for unitary groups, ii.
\newblock {\em Algebra \& Number Theory}, 5(7):923--1000, 2012.

\bibitem[{Sta}25]{stacks-project}
The {Stacks project authors}.
\newblock The stacks project.
\newblock \url{https://stacks.math.columbia.edu}, 2025.

\bibitem[vdG99]{van1999cycles}
Gerard van~der Geer.
\newblock Cycles on the moduli space of abelian varieties.
\newblock In {\em Moduli of Curves and Abelian Varieties: The Dutch Intercity Seminar on Moduli}, pages 65--89. Springer, 1999.

\bibitem[Xia22]{xia2022some}
Jiacheng Xia.
\newblock Some cases of kudla’s modularity conjecture for unitary shimura varieties.
\newblock In {\em Forum of Mathematics, Sigma}, volume~10, page e37. Cambridge University Press, 2022.

\bibitem[Zha09]{zhang2009modularity}
Wei Zhang.
\newblock {\em Modularity of generating functions of special cycles on Shimura varieties}.
\newblock PhD thesis, Columbia University New York, 2009.

\end{thebibliography}

\end{document}